\documentclass[aap,MSNbibl,seceqn,dvips]{arximspdf}
\usepackage{psfrag}
\usepackage{graphicx}

\doi{10.1214/12-AAP899} 
\volume{23}
\issue{6}
\pubyear{2013}
\firstpage{2238}
\lastpage{2258}

\makeatletter

\newcommand{\rrvert}{\vert}
\newcommand{\llvert}{\vert}
\newcommand{\eqref}[1]{(\ref{#1})}
\newcommand{\ol}[1]{{\overline {#1}}}

\newcommand{\s}{m}

\newtheorem{theo}{Theorem}[section]
\newtheorem{prop}[theo]{Proposition}

\newtheorem{lem}[theo]{Lemma}
\newtheorem{cor}[theo]{Corollary}

\newproclaim{defin}[theo]{Definition}
\newproclaim{rmk}[theo]{Remark}
\newproclaim{conv}[theo]{Convention}
\newproclaim{exs}{Example}
\newproclaim{ex1}{Example 1}
\newproclaim{ex2}{Example 2}
\newproclaim{assum}[theo]{Assumption}

\newcommand{\Rm}{\mathbb{R}}
\newcommand{\Pm}{\mathbb{P}}
\newcommand{\Fk}{F_{k}}
\newcommand{\Fl}{F_{\ell}}
\newcommand{\OO}{\mathcal{O}}
\newcommand{\PS}{P_{S}}
\newcommand{\RR}{\Rm}
\newcommand{\oH}{\ol{H}}
\newcommand{\minus}{\smallsetminus}
\newcommand{\argmin}{\operatorname{arg\,min}}

\makeatother

\begin{document}
\begin{frontmatter}

\title{Sticky central limit theorems on open books}
\runtitle{Sticky central limit theorems on open books}

\begin{aug}
\author[a]{\fnms{Thomas}~\snm{Hotz}\thanksref{m1}\ead[label=e1]{thomas.hotz@tu-ilmenau.de}},
\author[b]{\fnms{Stephan}~\snm{Huckemann}\thanksref{m2}\ead[label=e2]{huckeman@math.uni-goettingen.de}},
\author[c]{\fnms{Huiling}~\snm{Le}\ead[label=e3]{huiling.le@nottingham.ac.uk}},
\author[d]{\fnms{J.~S.}~\snm{Marron}\thanksref{m3}\ead[label=e4]{marron@email.unc.edu}},
\author[e]{\fnms{Jonathan~C.}~\snm{Mattingly}\thanksref{m4}\ead[label=e5]{jonm@math.duke.edu}},
\author[e]{\fnms{Ezra}~\snm{Miller}\corref{}\thanksref{m5}\ead[label=e6]{ezra@math.duke.edu}\ead[label=u6,url]{www.math.duke.edu/\textasciitilde ezra}},
\author[e]{\fnms{James}~\snm{Nolen}\thanksref{m6}\ead[label=e7]{nolen@math.duke.edu}},
\author[f]{\fnms{Megan}~\snm{Owen}\thanksref{m7}\ead[label=u8,url]{www.cs.uwaterloo.ca/\textasciitilde m2owen}},
\author[g]{\fnms{Vic}~\snm{Patrangenaru}\thanksref{m8}\ead[label=u9,url]{www.stat.fsu.edu/\textasciitilde vic}}
\and\break
\author[d]{\fnms{Sean}~\snm{Skwerer}\thanksref{m3}\ead[label=e10]{sskwerer@unc.edu}}
\thankstext{m1}{Supported by DFG Grant CRC~803.}
\thankstext{m2}{Supported by DFG Grants CRC~755 and HU~1575/2.}
\thankstext{m3}{Supported by NSF Grant DMS-08-54908.}
\thankstext{m4}{Supported by NSF Grant DMS-08-54879.}
\thankstext{m5}{Supported by NSF Grants DMS-04-49102${}={}$DMS-10-14112 and DMS-10-01437.}
\thankstext{m6}{Supported by NSF Grant DMS-10-07572.}
\thankstext{m7}{Supported by a desJardins Postdoctoral Fellowship in Mathematical Biology at the
University of California, Berkeley.}
\thankstext{m8}{Supported by NSF Grants DMS-08-05977 and DMS-11-06935.}

\runauthor{T. Hotz et al.}
\affiliation{%
Ilmenau University of Technology,
University of G\"ottingen,
University of Nottingham,
University of North Carolina at Chapel Hill,
Duke University,
Duke University,
Duke University,
University of Waterloo,
Florida State University
and
University of North Carolina at Chapel Hill}

\address[a]{T. Hotz\\
Institute of Mathematics\\
Ilmenau University of Technology\\
Weimarer Strasse 25, 98693 Ilmenau\\
Germany\\
\printead{e1}}
\address[b]{S. Huckemann\\
Institute for Mathematical Stochastics\\
University of G\"ottingen\\
Goldschmidtstrasse 7, 37077 G\"ottingen\\
Germany\\
\printead{e2}}
\address[c]{H. Le\\
School of Mathematical Sciences\\
University of Nottingham\\
University Park\\
Nottingham, NG7 2RD\\
United Kingdom\\
\printead{e3}}
\address[d]{J.~S. Marron\\
S. Skwerer\\
Department of Statistics \\
\quad and Operations Research\\
University of North Carolina\\
Chapel Hill, North Carolina 27599\hspace*{22pt}\\
USA\\
\printead{e4} \\
\phantom{E-mail:\ }\printead*{e10}}

\address[e]{J. C. Mattingly\\
J. Nolen\\
E. Miller\\
Mathematics Department\\
Duke University\\
Durham, North Carolina 27708\\
USA\\
\printead{e5}\\
\phantom{E-mail:\ } \printead*{e7} \\
\printead{u6}}
\address[f]{M. Owen\\
Cheriton School of Computer Science\hspace*{10pt}\\
University of Waterloo\\
Waterloo\\
ON N2L 3G1 Canada\\
\printead{u8}}
\address[g]{V. Patrangenaru\\
Department of Statistics\\
Florida State University\\
Tallahassee, Florida 32306\\
USA\\
\printead{u9}}
\end{aug}

\received{\smonth{2} \syear{2012}}
\revised{\smonth{9} \syear{2012}}

\begin{abstract}
Given a probability distribution on an open book (a metric space
obtained by gluing a disjoint union of copies of a half-space along
their boundary hyperplanes), we define a precise concept of when the
Fr\'echet mean (barycenter) is \emph{sticky}.  This nonclassical
phenomenon is quantified by a law of large numbers (LLN) stating that
the empirical mean eventually almost surely lies on the
(codimension~$1$ and hence measure~$0$) \emph{spine} that is the glued
hyperplane, and a central limit theorem (CLT) stating that the
limiting distribution is Gaussian and supported on the spine.  We also
state versions of the LLN and CLT for the cases where the mean is
nonsticky (i.e., not lying on the spine) and partly sticky (i.e., is, on the spine but not~sticky).
\end{abstract}

\begin{keyword}[class=AMS]
\kwd{60B99}
\kwd{60F05}
\end{keyword}
\begin{keyword}
\kwd{Fr\'echet mean}
\kwd{central limit theorem}
\kwd{law of large numbers}
\kwd{stratified space}
\kwd{nonpositive curvature}
\end{keyword}

\end{frontmatter}

\section*{Introduction}\label{sintro}

The mean of a finite set of points in Euclidean space moves slightly
when one of the points is perturbed.  This fluctuation is pervasive in
classical probabilistic and statistical situations.  In geometric
contexts, the barycenter (Fr\'echet mean \cite{Fr48}, $L^2$-minimizer, least
squares approximation), which minimizes the sum of the square
distances to a given set of points, generalizes the notion of mean.
Intuition from the Euclidean setting suggests that if the points are
randomly sampled from a well-behaved probability distribution on a
space $M$ of dimension $d+1$, then the random variable that is the
barycenter ought not be confined to a particular subspace of
dimension $d$ or less, if the distribution is generic.  While this
intuition has been made rigorous when $M$ is a manifold \cite{Jup88,HL96,BP05,Huc11},
it can fail when $M$ has certain types of
singularities, as we demonstrate here for an \emph{open book}~$\OO$: a
space obtained by gluing disjoint copies of a half-space along their
boundary hyperplanes; see Section \ref{sopenbook} for precise
definitions.

\begin{figure}

\includegraphics{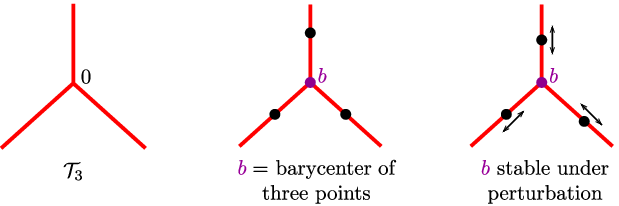}

\caption{(left) The space of rooted phylogenetic trees with three
leaves and fixed pendant edge lengths; (center) the probability
distribution supported on three points in $\mathcal{T}_3$ equidistant
from the vertex $0$ has bary\-center $0$; (right) perturbing the
distribution---and even macroscopically moving all three points a
limited distance---leaves the barycenter fixed.}\label{fT3}
\end{figure}

%
%
\begin{exs}\label{exT3}
The simplest singular space is the \emph{$3$-spider}: a union
$\mathcal{T}_3$ of three rays with their endpoints glued at a
point $0$ (Figure~\ref{fT3}, left).
This space $\mathcal{T}_3$ is the open book~$\OO$ of dimension $1$
with three leaves.  If three points are chosen equidistant from $0$ on
the different rays, then the barycenter lies at $0$ by symmetry
(Figure~\ref{fT3}, center).  The unexpected ``sticky'' phenomenon is
that wiggling one or more of the points has no effect on the
barycenter (Figure~\ref{fT3}, right).  For instance, if the points
lie at radius $r$ from $0$, then the barycenter remains at $0$ upon
moving one of the points to radius at most~$2r$.
\end{exs}

\begin{exs}\label{exdim2}
The name ``open book'' comes from the case of dimension 2, which looks
like an ordinary open book, in the usual lay sense of the words; see
Figure \ref{fd2K5}.
\end{exs}

\begin{figure}

\includegraphics{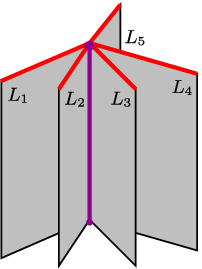}

\caption{Open book of dimension 2 with five leaves.  Ideally, the
picture of this embedding would continue to infinity vertically, both
up and down, as well as away from the spine on every
leaf.} \label{fd2K5}
\end{figure}

Our main goal is to define a precise concept of when a distribution on
an open book has a \emph{sticky mean} in Definition \ref{dcases}, and to quantify this highly
nonclassical condition with a law of large numbers (LLN) in
Theorem \ref{tstickyLLN} and a central limit theorem (CLT) in
Theorem \ref{tstickyCLT}.  Roughly speaking, the sticky LLN says that
in certain situations, empirical (sample) means almost surely
eventually lie on the \emph{spine}: the hyperplane shared by all of
the glued half-spaces by virtue of the gluing.  In Figure \ref{fT3},
the spine is the point $0$.  In Figure \ref{fd2K5}, the spine is the
central line.

The phenomenon of the sticky mean contrasts with the classical LLN,
where the empirical mean approaches the theoretical mean from all
directions.  The sticky CLT says that the limiting distribution is
Gaussian and supported on the spine.  Again, the nonclassical nature
of this result contrasts with the classical CLT, in which the limiting
distribution has full support rather than being supported on a thin
(positive codimension and hence measure zero) subset of the sample
space.  Versions of the LLN and CLT are also stated in
Theorems \ref{tstickyLLN}, \ref{tstickyCLT} and \ref{tnonsticky}
for the cases where the mean is:

\begin{itemize}
\item
nonsticky---not lying on the spine---so the LLN and CLT behave
classically; and

\item
partly sticky---on the spine but not sticky---so the LLN and CLT are
hybrids of the sticky and nonsticky ones.
\end{itemize}

This paper is motivated by a desire to understand statistical sampling
from topologically stratified spaces, including:

\begin{itemize}
\item
shape spaces, representing equivalence classes of point configurations
under operations such as rotation, translation, scaling, projective
transformations, or other nonlinear transformations (e.g., see \cite{DM98,PM03,PLS10}
for direct similarities, affine transformations,
and projective transformations, resp.);
\item
spaces of covariance matrices, arising as data points in diffusion
tensor imaging (see \cite{ArFiPeAy06,BassPier96,Schwartzman08,ScwMascTay08,BB11}, e.g.); and
\item
tree spaces, representing metric phylogenetic trees on fixed sets of
taxa (see \cite{BHV01,OwenProvan11,avgTree}, e.g.).
\end{itemize}

Open books are the simplest singular topologically stratified
spaces.\break
Roughly speaking, topologically stratified spaces decompose as finite
disjoint unions of manifolds (\emph{strata}) in such a way that the
singularities of the total space are constant along each stratum (this
is the structure described in \cite{SMT}, Section 1.4).  Every
topologically stratified space that is singular along a stratum of
codimension $1$ is, by definition of topological stratification,
locally homeomorphic to an open book along that stratum.  Therefore,
to understand statistical sampling from arbitrary stratified spaces
possessing singularities in maximal dimension, it is first necessary
to understand sampling from open books.

The metrics on open books that appear as local pieces of arbitrary
stratified spaces are arbitrary.  However, sticky means on open books
seem to stem from topological phenomena, rather than geometric ones,
so we consider only the simplest metric on $\OO$: each half-space has
the Euclidean metric and the boundaries are glued isometrically.
Although this restriction is substantial, these ``Euclidean'' open books
occur in applications.  For instance, the space $\mathcal{T}_3$ from
the first example above parametrizes all rooted (metric)
phylogenetic trees with three taxa and fixed pendant edge lengths.
More generally, open books of arbitrary dimension and precisely three
leaves reflect the local structure of phylogenetic tree space near any
point on a stratum of codimension $1$; such a point represents a tree
possessing a node with nonbinary branching.  Observations of
``unresolved'' (i.e., nonbinary) trees as barycenters of
biologically meaningful samples (see \cite{avgTree}, Examples 5.5 and 5.6, for descriptions of cases involving yeast
phylogenies and brain arteries) constituted crucial motivation for the
present study.

The relation between open books and tree spaces is that of local to
global.  After completing an early draft of this paper
we found that Basrak \cite{basrak} had independently and
simultaneously proved a sticky CLT for certain global situations in
dimension~$1$, namely arbitrary binary trees: connected graphs with no
cycles where each node is incident to at most three edges.  In
contrast, our dimension $1$ results are local, in that all edges meet,
but there can be more than three incident to the intersection.

It bears mentioning that in contrast to their behavior in open books,
barycenters do not stick to thin subspaces of shape spaces, or to thin
subspaces of more general quotients of manifolds by isometric proper
actions of Lie groups \cite{Huc12}.  The differentiating property
amounts to curvature: open books are, in a precise sense, negatively
curved at the spine, whereas passing to the quotient in the construction of shape
spaces adds positive curvature.  Basrak's binary trees \cite{basrak}
are negatively curved in the same way that open books or spaces
of trees are \cite{BHV01}: they are \emph{globally nonpositively
curved}.  (We recommend Sturm's exposition of this condition~\cite{sturm},
particularly for its clarity regarding connections
between probability and geometry, which was both a theoretical
starting point and a source of inspiration for our developments here.)
It is a principal long-term goal of our investigations to tease out
the connection between stickiness of means of probability
distributions with values in metric spaces and notions of negative
curvature.

%
%

\section{Open books}\label{sopenbook}

Set $S = \RR^d$, the real vector space of dimension $d$ with the
standard Euclidean metric.  If $\RR_{\geq 0} = [0,\infty)$ is the
closed nonnegative ray in the real line, then the closed half-space
\[
\oH_+ = \RR_{\geq 0} \times S
\]
is a metric subspace of $\RR^{d+1} = \RR \times S$ with boundary $S$
which we identify with $H = \{0\} \times S$, and interior $H_+ =
\RR_{>0} \times S$.  The \emph{open book} $\OO$ is the quotient of the
disjoint union $\oH_+ \times \{1,\ldots,K\}$ of $K$ closed half-spaces
modulo the equivalence relation that identifies their boundaries.
Therefore $p = (x,k) = (x^{(0)},x^{(1)},\ldots,x^{(d)}, k)$ is
identified with $q = (y,j) = (y^{(0)},y^{(1)},\ldots,y^{(d)}, j)$
whenever $x^{(0)} = 0 = y^{(0)}$ and $x^{(i)} = y^{(i)}$ for all $i
\in \{0,\ldots,d\}$, regardless of $k$ and $j$.  The following
definition summarizes and introduces terminology.

\begin{defin}[(Leaves and spine)]
The open book $\OO$ consists of $K \geq 3$ \emph{leaves} $L_k$, for $k
= 1, \dots, K$, each of dimension $d + 1$ and defined by
\[
L_k = \oH_+ \times \{k\}.
\]
The leaves are joined together along the \emph{spine} $L_0$ which
comprises the equivalence classes in $\bigcup_{k=1}^K (H \times
\{k\})$, that is, $L_0$ can be identified with the hyperplane
$H=\{0\}\times S$ or with the space $S = \RR^d$.  Thus, the open book
$\OO$ is the disjoint union
\begin{equation}
\label{lpartition} \OO = L_0 \cup L_1^+ \cup \cdots
\cup L_K^+
\end{equation}
of the spine $L_0$ and the interiors $L_k^+ = L_k \minus L_0$ of the
leaves, $k=1,\ldots,K$.  Figure~\ref{fd2K5} illustrates an open book
with $d = 1$ and $K = 5$.

When we speak of the spine in the following, we make clear which of
these three \emph{instances of the spine} we have in mind.  The
following diagram gives an overview of these instances, spaces and
mappings introduced further below in Definitions~\ref{dfolding-map}, \ref{dPS}, \ref{dconvproj} and in the proof of Lemma \ref{ly}.\vspace*{4pt}

\includegraphics{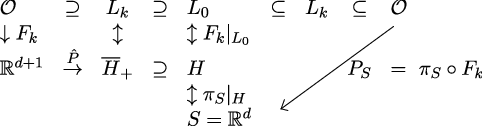}
\vspace*{4pt}

\end{defin}

\begin{defin}[(Reflection)]\label{defreflection}
For a given point $x \in \oH_+$, let $Rx \in \oH_- = \RR_{\leq 0}
\times \RR^d = (-\infty,0] \times \RR^d$ denote its
\textit{reflection} across the hyperplane $\oH_+ \cap \oH_- =
\{0\}\times S$.
\end{defin}

The metric $d$ on $\OO$ is expressed in terms of reflection in a
natural way: given two points $p, q \in \OO$, with $p = (x,k)$ and $q
= (y,j)$,
\begin{equation}
\label{eqd} d(p, q) = \cases{ |x - y|, & \quad $\mbox{if } k=j,$
\vspace*{2pt}\cr
|x - Ry|, &\quad  $\mbox{if } k\neq j,$}
\end{equation}
where $|x - y|$ denotes Euclidean distance on $\RR^{d+1}$.  Note that
if $k \neq j$ in equation~\eqref{eqd}, then $d(p,q) = 0$ if and only if
$x$ and $y$ lie on the spine and coincide.  Our assumption $K \geq 3$
implies that $\OO$ is not isometric to a subset of $\RR^{d+1}$ (as it
would be for $K \leq 2$).

The next lemma refers to \emph{globally nonpositive curvature}.  See
\cite{sturm} for a definition and background.  The only times we apply
this concept here are in noting the uniqueness of barycenters in our
context (see Definition \ref{dbarycenter} and the line following it)
and to obtain a quick proof of a strong law of large numbers
(Lemma \ref{lemLLN}).

\begin{lem}\label{lnpc}
The open book $(\OO,d)$ is a Hausdorff metric space that is globally
nonpositively curved, and its spine is isometric to $\RR^d$.
\end{lem}
\begin{pf}
\cite{sturm}, Example 3.3.
\end{pf}

%
%

\begin{rmk}\label{ractions}
Although the open book $\OO$ is not a vector space over $\RR$, scaling
by a positive constant $\lambda \in \RR_{\geq 0}$ is defined in the
natural way:
\[
\lambda p = (\lambda x,k) \qquad\mbox{for all } p = (x,k) \in \OO.
\]
The open book also carries an action of the spine $S$,
considered as an additive group, by translation, via the action of $S$
on each leaf:
\[
\OO \ni p = \bigl(x^{(0)},x^{(1)},\ldots, x^{(d)},k
\bigr) \stackrel{z} {\to} \bigl(x^{(0)},x^{(1)}+z^{(1)},
\ldots, x^{(d)}+z^{(d)},k\bigr)\in \OO,
\]
with $z = (z^{(1)},\ldots, z^{(d)}) \in S$.  For the above right-hand
side we write simply $z+p$.
\end{rmk}

\section{Probability measures on the open book}\label{smeasures}

Our goal is to understand the statistical behavior of points sampled
randomly from $\OO$.  Suppose that $\mu$ is a Borel probability
measure on $\OO$.  We assume throughout the paper that $d(0,q)$ has
bounded expectation under the measure $\mu$,
\begin{equation}
\label{finmean} \int_\OO \,d(0,q) \,d\mu(q) < \infty.
\end{equation}
When explicitly stated, we also assume the stronger condition
\begin{equation}
\label{finmean2} \int_\OO d(0,q)^2 \,d\mu(q) <
\infty,
\end{equation}
of square integrability.

\begin{lem}\label{lw}
Any Borel probability measure $\mu$ on the open book $\OO$ decomposes
uniquely as a weighted sum\vadjust{\goodbreak} of Borel probability measures $\mu_k$ on
the open leaves $L_k^+$ and a Borel probability measure $\mu_0$ on the
spine $L_0$.  More precisely, there are nonnegative real numbers
$\{w_k\}_{k=0}^K$ summing to $1$ such that, for any Borel set $A
\subseteq \OO$, the measure $\mu$ takes the value
\[
\mu(A) = w_0 \mu_0(A \cap L_0) + \sum
_{k=1}^K w_k
\mu_k\bigl(A \cap L_k^+\bigr).
\]
\end{lem}
\begin{pf}
This follows from the decomposition (\ref{lpartition}) and the
additivity of measures on disjoint sets.
\end{pf}

\begin{rmk}
For $k \geq 1$, $w_k=\mu(L_k^+)$ is the probability that a random
point lies in $L_k^+$, while $w_0=\mu(L_0)$ is the probability that a
point lies somewhere on the spine.
\end{rmk}

\begin{assum}\label{cnondegen}
Throughout this paper, assume the nondegeneracy condition
\begin{equation}
\label{nondegen} w_k=\mu \bigl(L_k^+ \bigr) > 0 \qquad\mbox{for all } k \in \{1,\ldots,K\}.
\end{equation}
Otherwise, we would remove those leaves $L_k$ for which $\mu(L_k^+) =
0$ from the open book.  Nondegeneracy implies that $w_0 < 1$ and $0 <
w_k < 1$ for all $k \geq 1$ in the decomposition from Lemma \ref{lw}.
\end{assum}

\begin{defin}[(Folding map)]\label{dfolding-map}
For $k \in \{1,\ldots,K\}$ the \emph{$k$\textup{th} folding map} $\Fk\dvtx  \OO
\to \RR^{d+1}$ sends $p \in \OO$ to
\[
\Fk p = \cases{x, &\quad  $\mbox{if } p = (x,k) \in
L_k,$
\vspace*{2pt}\cr
Rx, & \quad $\mbox{if } p = (x,j) \in L_j \mbox{ and } j \neq k,$}
\]
where the reflection operator $R$ was defined in
Definition \ref{defreflection}.
\end{defin}

\begin{rmk}
In the definition of the folding map $\Fk$, the leaf $L_k$ is
identified with the subset $\oH_+ \subset \RR^{d+1}$, by slight abuse
of notation (again).  The other leaves $L_j$ are collapsed to the
negative half-space $\oH_- \subset \RR^{d+1}$ via the reflection map.
All of these identifications have the same effect on the spine $S$,
which becomes the hyperplane $H = \{0\} \times \RR^d \subset
\RR^{d+1}$.  For example, $F_4$ takes the picture in
Figure \ref{fd2K5} to $\RR^2$ as in Figure \ref{fig3}.

\begin{figure}

\includegraphics{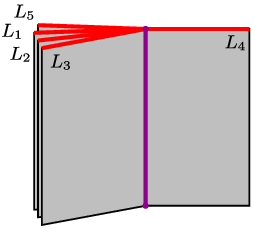}

\caption{The 4th folding map identifies leaf $L_4$ with the half-space
$\bar{H}_+ $and identifies all other leaves $L_j$ for $j \neq k$ with
the half-space $\bar{H}_-$.}\label{fig3}
\end{figure}

The notations $H_+$ and $H_-$ (with no bars) are reserved for the
\emph{strictly positive} and \emph{strictly negative} open half-spaces
that are the interiors of $\oH_+$ and $\oH_-$, respectively.
\end{rmk}

\begin{lem}
Under the folding map $\Fk$, the measure $\mu$ pushes forward to a
measure $\tilde\mu_k = \mu \circ \Fk^{-1}$ on $\RR^{d+1}$ such that,
given a Borel subset $A \subseteq \RR^{d+1}$,
\[
\tilde\mu_k(A) = w_k \mu_k(A \cap \oH_+)
+ w_0 \mu_0(A \cap S) + \mathop{\sum_{j \geq 1}}_{j \neq k} w_j \mu_j(A \cap H_-).
\]
\end{lem}

\begin{pf}
Lemma \ref{lw}.
\end{pf}

\begin{defin}[(First moment on a leaf)]\label{dFirstMomentOnLeaf}
Let $x^{(0)},\ldots,x^{(d)}$ be the coordinate functions
on $\RR^{d+1}$.  The \emph{first moment} of the measure $\mu$ on the
$k$th leaf $L_k$ is the real number
\[
\s_k = \int_{\RR^{d+1}} x^{(0)} \,d\tilde
\mu_k(x) = \int_\OO (\pi_0 \Fk p)
\,d\mu(p),
\]
where $\pi_0\dvtx \RR^{d+1} \to \RR$ is the orthogonal projection with
kernel $H = \{0\} \times \RR^d$.
\end{defin}

\begin{rmk}
For any point $p \in \OO$, the projection $\pi_0 \Fk p$ is positive if
$p \in L_k^+$ and negative if $p \in L_j^+$ for some $j \neq k$.
Moreover, $|\pi_0 \Fk p| = |x^{(0)}|$ is the distance of $p$ from the
spine.  The integrability in equation \eqref{finmean} guarantees that the
first moments of $\mu$ are all finite.
\end{rmk}

\begin{theo}\label{tcases}
Under integrability \eqref{finmean} and
nondegeneracy \eqref{nondegen}, either:
\begin{longlist}[(3)]
\item[(1)]
$\s_j < 0$ for all indices $j \in \{1,\ldots,K\}$,
or there is exactly one index $k \in \{1,\ldots,K\}$ such that $\s_k
\geq 0$, in which case either:
\item[(2)]
$\s_k > 0$, or
\item[(3)]
$\s_k = 0$.
\end{longlist}
\end{theo}

\begin{pf}
For $k = 1,\ldots,K$, let
\[
v_k = \int_{H_+} x^{(0)} \,d
\mu_k(x) .
\]
The nondegeneracy (\ref{nondegen}) implies that $v_k > 0$. Observe that
\[
\s_k = w_k v_k - \mathop{\sum
_{j \geq 1}}_{ j \neq k} w_j v_j.
\]
For any $j \neq k\in \{1,\ldots,K\}$,
\[
\s_j = w_j v_j - \mathop{\sum
_{\ell \geq 1}}_{ \ell \neq j} w_\ell v_\ell \leq
w_j v_j - w_k v_k \leq
\biggl(\mathop{\sum_{\ell \geq 1}}_{ \ell \neq k}
w_\ell v_\ell \biggr) - w_k v_k = -
\s_k,
\]
since the weights $w_\ell$ are nonnegative.  Therefore, if $\s_k >0$
for some $k$, then $\s_j \leq -\s_k < 0$ for all $j \neq k$.  Also, if
$\s_k = 0$ for some index $k$, then $\s_j \leq 0$ for all $j \neq k$.

Now suppose there are two indices $j,k \in \{1,\ldots,K\}$ such that
$j \neq k$ and $\s_j = 0$ and $\s_k = 0$. Then
\[
0 = \s_j = w_j v_j - w_k
v_k - \mathop{\sum_{\ell \geq 1}}_{ \ell \neq j, k}
w_\ell v_\ell
\]
and
\[
0 = \s_k = w_k v_k - w_j
v_j - \mathop{\sum_{\ell \geq 1}}_{ \ell \neq j, k}
w_\ell v_\ell.
\]
Adding these two equalities results in
\[
0 = \s_j + \s_k = - 2 \mathop{\sum
_{\ell \geq 1}}_{ \ell \neq j, k} w_\ell v_\ell.
\]
Since $w_\ell v_\ell \geq 0$, it follows that $w_\ell v_\ell = 0$ for
all $i \neq j,k$.  Consequently, $\mu(L_\ell^+) = 0$ for all $\ell
\neq j,k$.  However, this contradicts nondegeneracy \eqref{nondegen}
and the fact that $K \geq 3$.  Hence at most one of the numbers $\s_k$
can be nonnegative.
\end{pf}

Motivated by Theorem \ref{tstickyLLN} and
Corollary \ref{cstickylim}, we use the following terms to describe
the three mutually-exclusive conditions given in Theorem \ref{tcases}.

\begin{defin}\label{dcases}
Under integrability \eqref{finmean} and
nondegeneracy \eqref{nondegen}, we say that the mean of the
measure $\mu$ is either:
\begin{longlist}[(1)]
\item[(1)]
\emph{sticky} if $\s_j < 0$ for all indices $j \in \{1,\ldots,K\}$, or
\item[(2)]
\emph{nonsticky} if $\s_k > 0$ for some (unique) $k \in \{1,\ldots,K\}$, or
\item[(3)]
\emph{partly sticky} if $\s_k = 0$ for some (unique) $k \in \{1,\ldots,K\}$.
\end{longlist}
\end{defin}

\begin{rmk}\label{rsquareint}
If square integrability (\ref{finmean2}) also holds, the first moment
$m_k$ may be identified with the partial derivative
\[
m_k = - \frac{\partial \Gamma_k }{\partial x^{(0)}} (x) \Big|_{x^{(0)} = 0}
\]
where $\Gamma_k\dvtx \Rm^{d+1} \to \Rm$ is defined by
\[
\Gamma_k(x) = \frac{1}{2} \int_{\Rm^{d+1}} |x -
y|^2 \,d\tilde \mu_k(y).
\]
Observe that $-\frac{\partial \Gamma_k }{\partial x^{(0)}} (x)$
depends on $x^{(0)}$, but not on $(x^{(1)},\dots,x^{(d)})$.
\end{rmk}

\section{Sample means}\label{smeans}

For any finite collection of points $\{p_n\}_{n=1}^N \subset \OO$, the
Fr\'echet mean is a natural generalization of the arithmetic mean in
Euclidean space:

\begin{defin}\label{dbarycenter}
The \emph{Fr\'echet mean}, or \emph{barycenter}, of a set
$\{p_n\}_{n=1}^N \subset \OO$ of points is
\[
b(p_1,\ldots,p_N) = \mathop{\argmin}_{p \in \OO} \Biggl(\sum
_{n=1}^N d( p, p_n)^2
\Biggr).
\]
\end{defin}
By Lemma \ref{lnpc} and \cite{sturm}, Proposition 4.3, the barycenter
$b(p_1,\ldots,p_N) \in \OO$ exists and is unique.

\begin{defin}
For fixed $k \in \{1,\ldots,K\}$, the point $\eta_{k,N} \in \RR^{d+1}$
defined by
\begin{equation}
\label{etaNdef} \eta_{k,N} = \frac{1}{N} \sum
_{n=1}^N \Fk p_n
\end{equation}
is the \emph{$k$\textup{th} folded average}: the barycenter of the
pushforward under the $k$th folding map.
\end{defin}

For a set of points $\{p_n\}_{n=1}^N \subset \OO$, the condition
$b(p_1,\ldots,p_N) \in L_0$ does not necessarily imply $\eta_{k,N} \in
H$.  Nevertheless, the following lemma establishes an important
relationship between $b(p_1,\ldots,p_N)$
and $\eta_{k,N}$. Specifically, taking barycenters commutes with the
$k$th folding in two cases: if the barycenter lies off the spine
in $L_k^+$; or if the $k$th folded average lies in the closure of
the positive half-space.

\begin{lem}\label{lembNetaN}
Let $\{p_n\}_{n=1}^N \subset \OO$ and $b_N = b(p_1,\ldots,p_N)$. If
$b_N \in L_k^+$, then $\eta_{k,N} \in H_+$ and $\eta_{k,N} = \Fk b_N$.
If $\eta_{k,N} \in \oH_+$, then $b_N \in L_k$ and $\Fk b_N =
\eta_{k,N}$ (i.e. $b_N = (\eta_{k,N} ,k)$).
\end{lem}

\begin{pf}
Let $k,\ell \in \{1,\ldots,K\}$. If $p \in L_k$, then $d(p,p_n) =
|\Fk p - \Fk p_n|$.  Therefore, if $b_N \in L_k^+$, then
\[
b_N = \mathop{\argmin}_{p \in \OO} \sum_{n=1}^N
d( p, p_n)^2 = \mathop{\argmin}_{p \in L_k^+} \sum
_{n=1}^N |\Fk p - \Fk p_n|^2.
\]
Since $\Fk$ is continuously bijective from $L_k$ to $\oH_+$, this
implies that the function
\[
z \mapsto \sum_{n=1}^N |z - \Fk
p_n|^2
\]
attains a local minimum in the open set $H_+$.  However, this
functional has only one local minimizer, which must be the unique
global minimizer $\eta_{k,N}$,
\[
\eta_{k,N} = \mathop{\argmin}_{z \in \RR^{d+1}} \sum_{n=1}^N
|z - \Fk p_n|^2.
\]
Consequently, $\eta_{k,N} \in H_+$ and hence $\Fk b_N = \eta_{k,N}$.

If $b_N \notin L_k$, then $b_N \in L_\ell^+$ for some $\ell \neq k$.
Hence $\eta_{\ell,N} = \Fl b_N$, as we have shown.  In particular,
$\eta_{\ell,N} \in H_+$ and $\pi_0 \eta_{\ell,N} > 0$.  Hence
\begin{equation}
\label{FkFjneg}\quad  \sum_{p_n \in L_\ell^+} \pi_0 \Fl
p_n >- \sum_{p_n \notin L_\ell^+} \pi_0 \Fl
p_n \geq - \sum_{p_n \in L_k} \pi_0
\Fl p_n = \sum_{p_n \in L_k} \pi_0
\Fk p_n.
\end{equation}
Observe that
\begin{eqnarray*}
\pi_0 \eta_{k,N}& =& \frac{1}{N} \sum
_{n=1}^N \pi_0 \Fk p_n
\leq \frac{1}{N} \sum_{p_n \in L_k} \pi_0
\Fk p_n + \frac{1}{N} \sum_{p_n \in L_\ell^+}
\pi_0 \Fk p_n
\\
& = & \frac{1}{N} \sum_{p_n \in L_k} \pi_0
\Fk p_n - \frac{1}{N} \sum_{p_n \in L_\ell^+}
\pi_0 \Fl p_n.
\end{eqnarray*}
Because of equation \eqref{FkFjneg}, this last expression is negative.
Hence, we have shown that $b_N \notin L_k$ implies $\eta_{k,N} \in
H_-$.  Therefore, if $\eta_{k,N} \in \oH_+$ it must be that $b_N \in
L_k$.  Consequently, as above,
\begin{eqnarray*}
b_N & = & \mathop{\argmin}_{p \in \OO} \sum_{n=1}^N
d( p, p_n)^2
\\
& = & \mathop{\argmin}_{p \in L_k} \sum_{n=1}^N
|\Fk p - \Fk p_n|^2
\\
& = & \Fk^{-1} \Biggl( \mathop{\argmin}_{z \in \oH_+} \sum
_{n=1}^N |z - \Fk p_n|^2
\Biggr)
\\
& = & \Fk^{-1} \eta_{k,N}.
\end{eqnarray*}
Note that $\Fk^{-1}\eta_{k,N}$ is well defined, since $\eta_{k,N} \in
\oH_+$.
\end{pf}

\begin{defin}\label{dPS}
Given a point $p = (x,j) = (x^{(0)},x^{(1)},\ldots,x^{(d)},j) \in
\OO$,
\[
\PS p = \bigl(x^{(1)},\ldots,x^{(d)}\bigr) \in S
\]
is the orthogonal projection of $p$ onto the spine $S$.
\end{defin}

The following lemma shows that taking barycenters commutes with
projection to the spine.

\begin{lem}\label{ly}
If $\{p_n\}_{n=1}^N \subset \OO$ and
\[
\bar y_N = \frac{1}{N} \sum_{n=1}^N
\PS p_n,
\]
then $\bar y_N = \PS b(p_1,\ldots,p_N)$.
\end{lem}

\begin{pf}
Let $\pi_S\dvtx \RR^{d+1} \to \RR^d$ be the orthogonal projection onto the
last $d$ coordinates.  Let $b_N = b(p_1,\ldots,p_N)$.  If $b_N \in
L_k^+$ for some $k$, then $\eta_{k,N} = \Fk b_N$ by
Lemma \ref{lembNetaN}.  Therefore, since $\PS p = \pi_S \Fk p$ for
all $p \in \OO$,
\[
\PS b_N = \pi_S \Fk b_N =
\pi_S \eta_{k,N} = \frac{1}{N} \sum
_{n=1}^N \pi_S \Fk p_n =
\frac{1}{N} \sum_{n=1}^N \PS
p_n = \bar y_N.
\]

On the other hand, if $b_N \in L_0$ then by definition of $b_N$,
\[
b_N = \mathop{\argmin}_{p \in L_0} \sum_{n=1}^N
d(p,p_n)^2 = \mathop{\argmin}_{p \in S} \sum
_{n=1}^N \bigl(|\pi_0
p_n|^2 + |p - \PS p_n|^2
\bigr).
\]
Therefore $\PS b_N = \argmin_{y \in \RR^d} \sum_{n=1}^N |y - \PS
p_n|^2 = \frac{1}{N} \sum_{n=1}^N \PS p_n = \bar y_N$, as desired.
\end{pf}

\section{Random sampling and the law of large numbers}\label{sLLN}

We now consider points $\{p_n\}_{n=1}^N$ sampled independently at
random from a Borel probability measure $\mu$ on $\OO$; we wish to
understand the statistical behavior of their barycenter for large
$N$.  More precisely, let $(\Omega,\mathcal{F},\Pm)$ be a probability
space, and for each integer $n \geq 1$ let $p_n(\omega)\dvtx \Omega \to
\OO$ for fixed $\omega \in \Omega$ be a random point in $\OO$.

Assume for all $n \geq 1$ that $p_1,\ldots,p_n$ are independent random
variables and that for any Borel set $A \subseteq \OO$,
\[
\Pm(p_n \in A) = \Pm \bigl(\bigl\{\omega \in \Omega \mid
p_n(\omega) \in A\bigr\} \bigr) = \mu(A).
\]
The sample space $\Omega$ may be constructed as the set of infinite
sequences $(p_1,p_2, p_3,\ldots)$ of points in $\OO$ endowed with the
product measure $\Pm =\break \prod_{n=1}^\infty \mu(p_n)$ on the
$\sigma$-algebra $\mathcal{F}$ generated by cylinder sets.  Observe
that the folded points $\{\Fk p_n(\omega)\}_{n=1}^\infty \subset
\RR^{d+1}$ are independent, each distributed according to
$\tilde\mu_k$.

\begin{defin}
For any positive integer $N$, let $b_N(\omega) = b(p_1,\ldots,p_N)$
denote the barycenter of the random sample
$\{p_1(\omega),\ldots,p_N(\omega)\}$.  This random point in $\OO$ is
the \emph{empirical mean} of the distribution $\mu$. Similarly, for $k
\in \{1,\dots,K\}$, the random point $\eta_{k,N}(\omega) \in
\RR^{d+1}$ denotes the \emph{$k$\textup{th} folded average} of the random
sample $\{p_1(\omega),\ldots,p_N(\omega)\}$, as defined
by (\ref{etaNdef}).
\end{defin}

The goal is to understand the statistical behavior of empirical
means $b_N$ as \mbox{$N \to \infty$}.

\begin{lem}[(Strong law of large numbers)]\label{lemLLN}
There is a unique point $\bar b \in \OO$ such that
\[
\lim_{N \to \infty} b_N(\omega) = \bar b
\]
holds $\Pm$-almost surely.  If the square integrability condition (\ref{finmean2}) also holds, the limit $\bar b$ is the \emph{Fr\'echet
mean} (or \emph{barycenter}) of $\mu$,
\[
\bar b = \mathop{\argmin}_{p \in \OO} \int_\OO
d(p,q)^2 \,d\mu(q).
\]
\end{lem}
\begin{pf}
This is a special case of \cite{sturm}, Proposition 6.6, whose
generality occurs in the context of distributions on globally
nonpositively curved spaces.  (An elementary proof from scratch is
also possible, using arguments similar to the proof of
Theorem \ref{tstickyLLN}.  In general on metric spaces, there can be
more than one Fr\'echet mean, and there are corresponding set-valued
strong laws \cite{Zie77,BP03}.)
\end{pf}

\begin{theo}[(Sticky LLN)]\label{tstickyLLN}
Assume nondegeneracy \eqref{nondegen}.
\begin{longlist}[(1)]
\item[(1)]
If the moment $\s_j$ satisfies $\s_j < 0$, then there is a random
integer $N^*(\omega)$ such that $b_N(\omega) \notin L_j^+$ for all $N
\geq N^*(\omega)$ holds $\Pm$-almost surely. Furthermore, $\bar b
\notin L_j^+$.
\item[(2)]
If the moment $\s_k$ satisfies $\s_k > 0$, then there is a random
integer $N^*(\omega)$ such that $b_N(\omega) \in L_k^+$ for all $N
\geq N^*(\omega)$ holds $\Pm$-almost surely. Furthermore, $\bar b \in
L_k^+$.
\item[(3)]
If the moment $\s_k$ satisfies $\s_k = 0$, then there is a random
integer $N^*(\omega)$ such that $b_N(\omega) \in L_k$ for all $N \geq
N^*(\omega)$ holds $\Pm$-almost surely.  Furthermore, $\bar b \in
L_0$.
\end{longlist}
\end{theo}

\begin{pf}
By the usual strong law of large numbers,
\[
\lim_{N \to \infty} \eta_{k,N} = \bar\eta_k = \int
_{\RR^{d+1}} x \,d\tilde\mu_k(x)
\]
holds $\Pm$-almost surely.  Observe that $\s_k = \pi_0 \bar \eta_k$.
Therefore, if $\s_k > 0$, $\bar\eta_k \in H_+$ and $\eta_{k,N} \in
H_+$ for all sufficiently large $N$.  In that case, $b_N \in L_k^+$
for all sufficiently large $N$ by Lemma \ref{lembNetaN}.  In fact,
$\pi_0 b_N = \pi_0 \eta_{k,N} > \s_k/2 >0$ for $N$ sufficiently large,
so by\vadjust{\goodbreak} virtue of Lemma \ref{lemLLN}, $\bar b \in L_k^+$.  The same
argument starting with $m_k \geq 0$ proves the case $m_k = 0$.  On the
other hand, if $\s_j < 0$, then $\eta_{j,N} \in H_-$ for all
sufficiently large $N$; Lemma \ref{lembNetaN} implies that $b_N
\notin L_j^+$ for all sufficiently large $N$, and $\bar b \notin
L_j^+$.
\end{pf}

As a consequence, if the mean of $\mu$ is sticky then
the empirical mean $b_N$ sticks to the spine $L_0 \subset \OO$ for all
sufficiently large $N$, in the following sense.

\begin{cor}\label{cstickylim}
If the mean of $\mu$ is sticky, then there is a random integer
$N^*(\omega)$ such that $b_N(\omega) \in L_0$ for all $N \geq
N^*(\omega)$ holds $\Pm$-almost surely. Moreover, $\bar b \in L_0$. If
the mean of $\mu$ is partly sticky, with $\s_k = 0$, then then there
is a random integer $N^*(\omega)$ such that $b_N(\omega) \in L_k$ for
all $N \geq N^*(\omega)$ holds $\Pm$-almost surely.  Moreover, $\bar b
\in L_0$.
\end{cor}

Recall that $\PS$ is the orthogonal projection onto the spine $S$.  The
measure $\mu$ pushes forward along the projection to a measure $\mu_S
= \mu \circ \PS^{-1}$ on $S$,
\[
\mu_S(A) = \mu\bigl(\PS^{-1} A\bigr)
\]
for any Borel set $A \subseteq \RR^d$.  Note that $\mu_0(A) \leq
\mu_S(A)$ for all Borel sets $A \subseteq S$, but $\mu_S \neq \mu_0$
by Assumption \ref{cnondegen}.

\begin{cor}\label{cPSb}
In all cases (sticky, nonsticky, partly sticky), the limit $\bar b \in
\OO$ satisfies
\begin{equation}
\label{PSblim} P_S \bar b = \int_{S} y \,d
\mu_S(y).
\end{equation}
\end{cor}
\begin{pf}
By Lemma \ref{ly} and Theorem \ref{tstickyLLN},
\[
P_S \bar b = P_S \lim_{N \to \infty} b_N =
\lim_{N \to \infty} \bar y_N
\]
holds almost surely.  By the strong law of large numbers for $\bar y_N
\in S = \RR^d$, the last limit is (\ref{PSblim}).
\end{pf}

\section{Central limit theorems}\label{sCLTs}

In this section we consider fluctuations of the empirical mean
$b_N(\omega)$ about the asymp\-totic limit $\bar b$, within the tangent
cone at $\bar b$.  We have shown that if the mean is either sticky or
partly sticky, then $\bar b \in S$, and the tangent cone at $\bar b$
is an open book $\OO$.  On the other hand, if the mean is
nonsticky, with $\s_k > 0$, then $\bar b$ is in the interior of the
leaf $L_k^+$ and the tangent cone at $\bar b$ is the vector space
$\RR^{d+1}$.  We treat these two scenarios separately.

These facts essentially follow from Theorem \ref{tstickyLLN} which shows
that in the sticky cases with probability one the fluctuations away
from the mean in certain directions stop as more random variables are
added to the empirical mean.  In particular, this implies that the
correctly normalized limit of the fluctuation from the mean cannot, in
the sticky case, converge to a Gaussian random variable as one would\vadjust{\goodbreak}
have in the standard central limit theorem.  Since the fluctuations in
some directions are exactly zero at some point along each sequence of
random variables, it is not all together surprising that limiting
measure has mass concentrated on a lower dimensional set.  This is the
content of Theorem \ref{tstickyCLT} which is the principal result of
this section.\looseness=-1

\subsection{The sticky central limit theorem}

Throughout this section, assume $m_j \leq 0$ for all $j
\in\{1,\ldots,K\}$.  Hence $\bar b \in L_0$, and the mean is either
sticky or partially sticky.  In the partially sticky case, denote
by $k$ the unique index satisfying $m_k = 0$.  The central limit
theorem involves a centered and rescaled empirical mean.

\begin{defin}[(Rescaled empirical mean)]\label{drescaled}
Assume that $\PS \bar b = 0$ (after the action of $-\PS \bar b\in S$
on $\OO$ as explained in Remark \ref{ractions} if necessary).  The
\emph{rescaled empirical mean} is the random variable $\sqrt{N} b_N
\in \OO$.  Write $\nu_N$ for its induced probability law on $\OO$,
\[
\Pm \bigl(\bigl\{\omega \mid \sqrt{N} b_N(\omega) \in A\bigr\}
\bigr) = \int_{\OO \cap A} \,d \nu_N(p)
\]
for all Borel sets $A \subseteq \OO$.
\end{defin}

Since in sticky settings, we need to collapse fluctuations in some
directions back to the spine, it is convenient to define the following
projection.
\begin{defin}\label{dconvproj}
The convex projection $\hat P$ of $\RR^{d+1}$ onto $\oH_+$ is
\[
\hat P x = \cases{\bigl(0,x^{(1)},
\ldots,x^{(d)}\bigr), & \quad $\mbox{if } x^{(0)} < 0,$
\vspace*{2pt}\cr
\bigl(x^{(0)},x^{(1)},\ldots,x^{(d)}\bigr), &\quad  $\mbox{if
} x^{(0)} \geq 0.$}
\]
\end{defin}

We now define measures which we will see shortly describe the limiting
behaviors of $\nu_N$ as $N\rightarrow \infty$. In short, they are the
limiting measures in the central limit theorem given in
Theorem \ref{tstickyCLT} below.

\begin{defin}\label{dspinal}
Assume square integrability \eqref{finmean2} and assume that $\PS \bar
b = 0$.
\begin{longlist}[(3)]
\item[(1)]
The \emph{spinal limit measure} $g_S$ is the law of a multivariate
normal random variable on the spine $S \cong \RR^d$ with mean zero and
covariance matrix
\[
C_S = \int_{\RR^d} y y^T \,d
\mu_S(y) = \int_\OO (\PS p) (\PS
p)^T \,d \mu(p).
\]

\item[(2)]
The $k$th \emph{costal
\setcounter{footnote}{8}\footnote{Adjective: of or pertaining to the ribs, in anatomy.}
limit measure} $g_k$ is the law of a multivariate normal random
variable on $\RR^{d+1}$ with mean zero and covariance matrix
\[
C_k = \int_{\RR^{d+1}} xx^T \,d \tilde
\mu_k(x) = \int_\OO (\Fk p) (\Fk
p)^T \,d \mu(p).
\]\eject

\item[(3)]
The $k$th \emph{spinocostal\footnote{Adjective: spanning the ribs and spine, in anatomy.}
limit measure} $h_k$ on the closed leaf $L_k \cong \oH_+$ is defined
by
\[
h_k(A) = h_k^{0} \bigl(F_k(A)
\cap H \bigr) + g_k \bigl(F_k(A) \cap \oH_+ \bigr)
\]
for Borel sets $A \subseteq L_k$, where the \emph{semispinal limit
measure} $h_k^{0}$ on $L_0$ is defined by
\[
h_k^0 \bigl((P_S|_{L_0})^{-1}B
\bigr) = g_S(B) - g_k \bigl((0,\infty) \times B \bigr)
\]
for Borel sets $B\subseteq S$.  (A possibly more natural definition of
$h_k$ is given in Proposition \ref{pconvexProj} below.)
\end{longlist}
\end{defin}

\begin{rmk}\label{rfinite}
Square integrability \eqref{finmean2} implies that the covariance
matrices are finite.
\end{rmk}

\begin{rmk}
The semispinal limit measure is generally not Gaussian.  Although
the orthogonal projection to $\RR^d$ of any Gaussian measure on
$\RR^{d+1}$ is Gaussian, $h_k^{0}$ is the projection of only half of a
Gaussian; this is implied by Proposition~\ref{pconvexProj}, an
alternate direct description of $h_k$ interpolating between the first
two parts of Definition \ref{dspinal}.
\end{rmk}

\begin{prop}\label{pconvexProj}
The spinocostal limit measure is the pushforward of the costal limit
measure $g_k$ under convex projection: $h_k = g_k \circ \hat
P^{-1}\circ F_k$.
\end{prop}
\begin{pf}
Since the measures agree on $L_k$ outside of $L_0$ by definition, it
is enough to show that
\begin{equation}
\label{g1g0P} h_k^0 \bigl((P_S|_{L_0})^{-1}B
\bigr) = g_k \bigl(\hat P^{-1}\circ (\pi_S|_H)^{-1}B
\bigr)
\end{equation}
for any Borel set $B \subseteq S$.  For any vectors $w, w' \in
\RR^{d+1}$ that lie on the spine $H \subseteq \RR^{d+1}$, considering
them as vectors in $z=\pi_S(w), z'=\pi_S(w') \in S=\RR^d$ results in
quantities $z^T C_S z'$, and $w^T C_k w'$.  The integrals in
Definition \ref{dspinal} directly imply that $z^T C_S z' = w^T C_k
w'$.  Consequently, the matrix $C_S$ is a submatrix of $C_k$; the
action of $C_k$ on the subspace $H$ is given by $C_S$.  Thus $g_S(B) =
g_k ((-\infty,\infty)\times B )$, and hence by definition
\begin{eqnarray*}
h_k^0(B) &=& g_k \bigl((-\infty,\infty)
\times B \bigr) - g_k \bigl((0,\infty)\times B \bigr) =
g_k \bigl((-\infty,0]\times B \bigr) \\
&=& g_k \bigl(\hat
P^{-1}\circ(\pi_S|_H)^{-1}B \bigr)
\end{eqnarray*}
for any Borel set $B \subseteq S$.
\end{pf}

Now we come to the primary result in the paper: as the sample size $N$
becomes large, the law $\nu_N$ of the rescaled empirical mean
converges weakly to the appropriate measure from
Definition \ref{dspinal}, according to how sticky the mean is.  (We
have included a forward reference to the nonsticky case in
Theorem \ref{tstickyCLT} to preserve the numbering of items 1, 2
and 3, which corresponds precisely to the numbering elsewhere, namely
Theorem \ref{tcases}, Definition \ref{dcases},
Theorem \ref{tstickyLLN}, and Definition~\ref{dspinal}.)  When the
mean is:
\begin{longlist}[(1)]
\item[(1)]
sticky, $\nu_N$ converges weakly to the spinal limit measure $g_S$;

\item[(2)]
nonsticky, $\nu_N$ converges weakly to the costal limit measure $g_j$
supported on the tangent space $\RR^{d+1}$ to the leaf $L_j$
containing the mean;

\item[(3)]
partly sticky, $\nu_N$ converges weakly to the spinocostal limit
measure $g_j$ supported on the (unique) leaf $L_k$ with moment $\s_k =
0$.
\end{longlist}
As discussed at the start of the section, the fact that the limiting
distribution is supported on the spine $S$ when the mean is sticky
follows from Theorem \ref{tstickyLLN}, since then the first moments
$\s_j$ are strictly negative for all $j$.

\begin{theo}[(Sticky CLT)]\label{tstickyCLT}
Let $\mu$ be a nondegenerate \eqref{nondegen} probability distribution
on the open book $\OO$ with finite second moment \eqref{finmean2}.
\begin{longlist}[(1)]
\item[(1)]
If the mean of $\mu$ is sticky, then for any continuous, bounded
function $\phi\dvtx \OO \to \RR$,
\[
\lim_{N \to \infty} \int_\OO \phi(p) \,d
\nu_N(p) = \int_S \phi\circ
(P_S|_{L_0})^{-1}(q) \,dg_S(q).
\]

\item[(2)]
If the mean of $\mu$ is nonsticky, then see Theorem \ref{tnonsticky}.

\item[(3)]
If the mean of $\mu$ is partly sticky, with first moment $\s_k = 0$,
then for any continuous bounded function $\phi\dvtx \OO \to \RR$,
\[
\lim_{N \to \infty} \int_\OO \phi(p) \,d
\nu_N(p) = \int_{\oH_+} \phi\circ
F_k^{-1}(q) \,dh_k(q).
\]
\end{longlist}
\end{theo}

\begin{pf}
The proof works by decomposing the relevant measures---the empirical
mean on the open book and its pushforward to $\RR^{d+1}$ under
folding---into pieces corresponding to the leaves and the spine.

Suppose that the mean is partly sticky with first moment $\s_k = 0$.
Let $\eta_N = \eta_{k,N}$ as in \eqref{etaNdef}, and let
$\nu_{\eta,N}(x)$ denote the law of $\sqrt{N} \eta_N$ on $\RR^{d+1}$.
By Lemma~\ref{lembNetaN}, $\nu_N(A) = \nu_{\eta,N}(F_k A)$ for any
Borel set $A \subseteq L_k$, and if $\phi$ is a continuous and bounded
function, then
\begin{eqnarray*}
\int_\OO \phi(p) \,d\nu_N(p) & = & \int
_{L^+_k}\phi(p) \,d\nu_N(p) + \int
_{\OO\minus L_k^+}\phi(p) \,d\nu_N(p)
\\
& = & \int_{H_+} \phi \bigl( \bigl(F_k^{-1}|H_+
\bigr)^{-1}(y) \bigr) \,d \nu_{\eta,N}(y) + \int
_{\OO \minus L_k^+} \phi(p) \,d \nu_N(p).
\end{eqnarray*}
The standard CLT in $\RR^{d+1}$ (e.g., \cite{Brei}, Theorem 11.10) implies
that the random variable $\sqrt{N} \eta_N$ converges in distribution
to a centered Gaussian with covariance~$C_k$.  Therefore,
\[
\lim_{N \to \infty} \int_{H_+} \phi \bigl(
\bigl(F_k^{-1}|H_+\bigr)^{-1}(y) \bigr) \,d
\nu_{\eta,N}(y) = \int_{H_+} \phi \bigl(
\bigl(F_k^{-1}|H_+\bigr)^{-1}(y) \bigr)
\,dg_k(y).
\]

\begin{lem}\label{lLminus}
If the $j$th first moment satisfies $\s_j < 0$, then $\nu_N(L_j^+)
\to 0$ and
\[
\lim_{N \to \infty} \int_{L_j^+} \phi(p) \,d
\nu_N(p) = 0.
\]
\end{lem}

\begin{pf}
Theorem \ref{tstickyLLN}(1).
\end{pf}

Resuming the proof of the theorem, consider the term
\[
\int_{\OO\minus L_k^+} \phi(p) \,d \nu_N(p) = \int
_{L_0} \phi(p) \,d \nu_N(p) + \int
_{L_k^-} \phi(p) \,d \nu_N(p),
\]
where $L_k^- = \OO \minus L_k = \bigcup_{j \neq k} L_j^+$, which
excludes the spine $L_0$.  With the projection $P_0 \dvtx \OO \to L_0,
(x^{(0)},x,j) \mapsto (0,x,j)$ the function $p \mapsto \phi(P_0 p)$ is
again continuous and bounded, Lemma \ref{lLminus} implies that
\begin{equation}
\label{eLminus} \lim_{N \to \infty} \int_{L_k^-}
\phi(P_0 p) \,d \nu_N(p) = 0.
\end{equation}
Therefore,\vspace*{-1pt}
\begin{eqnarray*}
\lim_{N \to \infty} \int_{L_0} \phi(p) \,d
\nu_N(p) & = & \lim_{N \to  \infty} \int_{L_0}
\phi(P_0 p) \,d \nu_N(p)
\\[-1pt]
& = & \lim_{N \to \infty} \biggl( \int_{L_k^-}
\phi(P_0 p) \,d \nu_N(p) + \int_{L_0}
\phi(P_0 p) \,d \nu_N(p) \biggr)
\\[-1pt]
& = & \lim_{N \to \infty} \biggl( \int_\OO
\phi(P_0 p) \,d \nu_N(p) - \int_{L^+_k}
\phi(P_0 p) \,d \nu_N(p) \biggr).\vspace*{-1pt}
\end{eqnarray*}
Observe that\vspace*{-1pt}
\[
\int_\OO \phi(P_0 p) \,d \nu_N(p) =
\int_S \phi \circ (P_S|_{L_0})^{-1}(y)
\,d \gamma_N(y),\vspace*{-1pt}
\]
where $\gamma_N = \nu_N\circ P_S^{-1}$ which is the law of $\sqrt{N}
\bar y_N$ on $S$, where $\bar y_N$ is the projected barycenter from
Lemma \ref{ly}.  Therefore, setting $\hat \phi = \phi \circ
(P_S|_{L_0})^{-1}$ and applying the usual CLT to $\sqrt{N} \bar y_N
\in \RR^d$,\vspace*{-1pt}
\[
\lim_{N \to \infty} \int_\OO \phi(P_0 p) \,d
\nu_N(p) = \lim_{N \to \infty} \int_S \hat\phi(
y) \,d \gamma_{N}(y) = \int_S \hat\phi(y) \,d
g_S(y).
\]
We cannot apply the same argument to\vspace*{-1pt}
\[
\lim_{N \to \infty} \int_{L^+_k} \phi(P_0 p) \,d
\nu_N(p) = \lim_{N \to \infty} \int_{L^+_k} \hat
\phi(y) \,d \tau_N(y)\vspace*{-1pt}
\]
with $\tau_N = \nu \circ (P_S|_{L_k^+})^{-1}$ because there is no CLT
for $\tau_N$.  We have, however, above derived a CLT for $\nu_N \circ
F_k^{-1} = \nu_{\eta,N}$ on $H_+=F_k(L_k^+)$:\vspace*{-1pt}
\[
\lim_{N \to \infty} \int_{L^+_k} \phi(P_0 p) \,d
\nu_N(p) = \lim_{N \to \infty} \int_{H_+}
\widetilde \phi (q) \,d \nu_{\eta,N}(q) = \lim_{N \to \infty} \int
_{H_+} \widetilde\phi(q) \,d g_k(q),\vspace*{-1pt}
\]
where $\widetilde \phi = \phi\circ P_0 \circ F_k\circ\hat P^{-1}$.  In
summary, we have shown that\vspace*{-1pt}
\begin{eqnarray*}
&&\lim_{N \to \infty} \int_\OO \phi(p) \,d
\nu_N(p) \\[-1pt]
&&\qquad =  \int_{H_+} \phi\circ
F_k^{-1}(q) \,dg_k(q) + \int
_S \hat \phi(y) \,d g_S(y) - \int
_{H_+} \widetilde \phi(q) \,d g_k(q)
\\[-1pt]
&&\qquad =  \int_{H_+} \phi\circ F_k^{-1}(q)
\,dg_k(q) + \int_H \phi\circ
F_k^{-1}(q) \,dh_k^0(q)
\\[-1pt]
&&\qquad =  \int_{\oH_+} \phi\circ F^{-1}(q)
\,dh_k(q),
\end{eqnarray*}
where the second equality uses the fact that $\widetilde \phi =
\phi\circ F_k^{-1}$ on $H$ and the final equality the fact that $g_k$
has no mass supported on the spine $H$, so the integral of $\phi\circ
F^{-1} \,dg_k$ over $H_+$ can just as well be taken over $\oH_+$.

The sticky case proceeds in much the same way as the partly sticky
case does, except that instead of equation \eqref{eLminus}, the simpler
statement
\[
\lim_{N \to \infty} \int_{\OO \minus S} \phi(P_0 p) \,d
\nu_N(p) = 0
\]
holds. From that, the next step results in
\[
\lim_{N \to \infty} \int_{L_0} \phi(p) \,d
\nu_N(p) = \lim_{N \to \infty} \int_\OO
\phi(P_0 p) \,d \nu_N(p),
\]
and then the usual CLT applied to $\sqrt{N} \bar y_N \in \RR^d$ proves
the desired result.
\end{pf}

\subsection{The nonsticky central limit theorem}

If the mean is nonsticky with first moment $\s_k > 0$, then the limit
$\bar b$ is in the interior of $L_k^+$.  In this case, the tangent
cone at $\bar b$ is the vector space $\RR^{d+1}$, and the fluctuations
of $b_N$ about the limit $\bar b$ are qualitatively similar to what is
described in the classical central limit theorem.

\begin{defin}
In this section we let $\widetilde \nu_N$ be the law on $\RR^{d+1}$ of
the random variable $\sqrt{N} (F_{k} b_N - F_{k} \bar b)$,
\[
\Pm \bigl( \bigl\{\omega | \sqrt{N}(F_{k} b_N -
F_{k} \bar b) \in A \bigr\} \bigr) = \widetilde \nu_N(A)
\]
for all Borel sets $A \subseteq \RR^{d+1}$.
\end{defin}

\begin{defin}
Assume $\s_k > 0$.  Let $\tilde g_k$ be the law of a multivariate
normal random variable on $\RR^{d+1}$ with mean zero and covariance
matrix
\[
\tilde C_k = \int_{\RR^{d+1}} (x - F_k
\bar b) (x - F_k \bar b)^T \,d \tilde
\mu_k(x).
\]
\end{defin}

In contrast to the case of a sticky or partly sticky mean, the weak limit
of $\nu_N$ is that of a nondegenerate Gaussian on $\RR^{d+1}$:

\begin{theo}[(Nonsticky CLT)]\label{tnonsticky}
Assume $\s_k > 0$.  Then for any continuous bounded function
$\phi\dvtx \mathbb{R}^{d+1} \to \RR$,
\[
\lim_{N\to\infty} \int_{\mathbb{R}^{d+1}}\phi(x) \,d \widetilde
\nu_N(x) = \int_{\RR^{d+1}} \phi(x) \,d\tilde
g_k(x).
\]
\end{theo}

\begin{pf}
Since $\s_k > 0$, $\bar b \in L_k^+$ and Lemma \ref{lembNetaN}
implies $F_k \bar b = \bar \eta =\break  \int_{\RR^{d+1}} x  \,d\tilde
\mu_k(x)$.  Also,
\[
\sqrt{N}\bigl(F_{k} b_N(\omega) - F_{k} \bar b
\bigr) = \sqrt{N}\bigl(\eta_{k,N}(\omega) - \bar \eta\bigr) \qquad\forall N
\geq N^*(\omega)
\]
holds with probability one.  Therefore, for any Borel set
\[
\bigl\llvert \widetilde \nu_N(A) - \Pm \bigl( \bigl\{\omega |
\sqrt{N}\bigl(\eta_{k,N}(\omega) - \bar \eta\bigr) \in A \bigr\} \bigr)
\bigr\rrvert \leq R_N,
\]
where $R_N = \Pm ( \{\omega  |  N < N^*(\omega) \}  )$.  By
the classical central limit theorem, the random variable
$\sqrt{N}(\eta_{k,N}(\omega) - \bar \eta)$ converges in law to a
centered, multivariate Gaussian on $\RR^{d+1}$ with covariance $C_k$
as $N \to \infty$.  Consequently,
\[
\limsup_{N \to \infty} \biggl\llvert \int_{\mathbb{R}^{d+1}} \phi(x) \,d
\widetilde \nu_N(x) - \int_{\RR^{d+1}} \phi(x) \,d\tilde
g_k(x) \biggr\rrvert \leq 2\limsup_{N \to \infty} R_N
\Vert\phi\Vert_\infty = 0.
\]
\upqed\end{pf}

\section*{Acknowledgments}
Our thanks go to Seth Sullivant, for bringing our attention to
stickiness for means of phylogenetic trees.  This paper is output from
a Working Group that ran under the Statistical and Applied
Mathematical Sciences Institute (\textsc{SAMSI})
2010--2011 program on Analysis of Object Data. We are grateful to the
other members of the \textsc{SAMSI} Working Group on sampling from
stratified spaces, not listed among the authors, who facilitated the
development of this research program.  We are indebted to
\textsc{SAMSI} itself, for sponsoring many of the authors' visits to
the Research Triangle, for hosting the Working Group meetings, and for
stimulating this cross-disciplinary research in its unique way.


\printaddresses

\end{document}